\documentclass[a4paper,10pt]{article}
\usepackage[latin1]{inputenc}
\usepackage{amsmath, amscd, amsthm, amssymb, amsfonts}
\usepackage{xy}
\usepackage{multirow,array}

\newtheorem{lemma}{Lemma}[section]
\newtheorem{theorem}[lemma]{Theorem}
\newtheorem{proposition}[lemma]{Proposition}
\newtheorem{corollary}[lemma]{Corollary}

\newtheorem{example}{Example}

\DeclareMathOperator{\ini}{in}

\DeclareMathOperator{\spann}{span}

\DeclareMathOperator{\nf}{Nf}

\DeclareMathOperator{\supp}{supp}
\DeclareMathOperator{\numbits}{numbits}
\DeclareMathOperator{\Ess}{Ess}

\DeclareMathOperator{\Varord}{Varord}

\DeclareMathOperator{\Occ}{cp}
\DeclareMathOperator{\ov}{ov}

\title{Complexity of comparing monomials and two improvements of the BM-algorithm}
\author{Samuel Lundqvist}
\date{}
\begin{document}

\maketitle

\begin{abstract}
We give a new algorithm for merging sorted lists of monomials. 
Together with a projection technique 
we obtain a new complexity bound for the BM-algorithm.
\end{abstract}

\section{Introduction} \label{sec:intro}

Vanishing ideals of points are of interest in many fields of mathematics --- 
they are used in coding theory, in interpolation problems and even in statistics. 
Recently, the vanishing ideal of a set of affine points has been studied in molecular biology \cite{LaubStigler04}. 

The BM-algorithm \cite{BM} was proposed as a tool to make computations over vanishing ideals of points.
When doing complexity studies of the BM-algorithm, one has to deal with 
arithmetic operations over the ground field and monomial manipulations. The number of arithmetic operations is reported \cite{FGLM,MMM} to be proportional to
$nm^3$, where $n$ denotes the number of variables and $m$ the number of points. The number of integer comparisons needed for the 
monomial manipulations is reported \cite{FGLM,MMM} to be proportional to $n^2m^2$. In the biological applications, 
the coefficients of the points takes  values in a finite field $\mathbb{Z}_p$ and one usually has $m \ll n$. 
Accordingly, one has begun searching for algorithms which are optimized for these situations.
In \cite{JustStigler06}, an algorithm which uses $O(nm^2+m^6)$ operations (arithmetic and integer operations are treated the same) is given, 
while in \cite{JustStigler07}, the same authors sharpened this bound so that it reads $O(nm^2 + pm^4 + pm^3 \log(pm))$ operations, 
where again, arithmetic and integer operations are treated the same.

In this paper, we first make a thorough study of the complexity of comparing monomials. We restrict our analysis to the admissible monomial orders on 
$n$ indeterminates given by invertible matrices with $\mathbb{Z}$-coefficients. These orders associate to each monomial an $n$-vector of integers with the property 
that comparing two monomials is the same as comparing the 
$n$-vectors lexicographically. Although this is a restriction on the set of admissible monomial orders, we remark that earlier complexity studies 
\cite{FGLM, JustStigler07,MMM} have been performed only on a much smaller set of of admissible orders, e.g. lex, deglex and degrevlex. 

To give bounds for the monomial manipulations, we study algorithms for comparing lexicographically sorted $n$-vectors 
and give a fast algorithm for merging lexicographically sorted lists of such vectors. 
Summation of ordered polynomials is one example of a situation where merge algorithms are used and thus, 
our merge algorithm could be used to speed up the computation of 
$S$-polynomials during the computations of a Gr\"{o}bner basis with respect to an ideal given by generators. 

The rest of the paper is focused on the BM-algorithm. We notice that the upper bound for the number of arithmetic operations during 
the BM-algorithm as given in \cite{FGLM} can be sharpened to read 
$$O(nm^2 + \min(m,n)m^3).$$
After introducing a projection technique we show that the upper bound for the time complexity of the monomial manipulation part can be lowered to 
$$O(\min(m,n)m^2 \log(m))$$ 
using one of the monomial orders lex, deglex or degrevlex. The factor $\log(m)$ comes from the model we use --- that adding two integers bounded by $m$ has time complexity $O(\log(m))$. It follows that the bottleneck of the BM-algorithm are the arithmetic operations. 

Our method has better complexity than 
both the standard BM-algorithm and the methods optimized for the situations where $m \ll n$. It is also more general than the methods in 
\cite{JustStigler06, JustStigler07} since we do not assume that $\Bbbk$ is a finite field. 

Finally, we use our methods to give new complexity results for the FGLM-algorithm \cite{FGLM} and for the algorithms concerning ideals defined by functionals given in \cite{MMM}.

Throughout the rest of the paper, we let $S=\Bbbk[x_1, \ldots, x_n]$ denote the polynomial ring in $n$ variables over a field $\Bbbk$. The notion $x^{\alpha}$, where
$\alpha =(\alpha_1, \ldots, \alpha_n)$ will be used as short for $x_1^{\alpha_1} \cdots x_n^{\alpha_n}$. 

\section{Monomial manipulations} \label{sec:mon}
Comparisons of monomials is a crucial part in the complexity analysis of computational algebra. 
However, different complexity models leads to different complexity results and this is problematic when it comes to comparing 
different algorithms. For instance, in \cite{FGLM,JustStigler07}, comparison of two monomials in $n$ variables  with respect to a monomial order 
is assumed to have time complexity $O(n)$, while in  \cite{Abbottetal,MMM}, the comparison of two monomials is assumed to 
be $O(n)$ integer comparisons. If one assumes an integer comparison to be made in constant time, then these complexities agree. But this is an unrealistic assumption,
since the time needed for comparing integers is dependent on the size of the integers. 

After fixing a sound computational model, we will show that for a common class of monomial orders, 
it is possible to compare two monomials of degree bounded by $m$ in time proportional to $n\max(\log(m/n),1)$. 
We then show that it is possible to merge
two lexicographically sorted lists $a$ and $b$ of $n$-tuples in some set $\Sigma$ with $s$ and $t$ elements respectively using at most 
$\min(s,t)+n$ comparisons in $\Sigma$ plus time proportional to
$\max(s,t)\log(n)$. This is better than the expected $\max(s,t)n$ $\Sigma$-comparisons. 
Combining these result, we obtain a method which merges two lists of monomials, 
sorted with respect to lex, deglex or degrevlex, in time proportional to $\min(s,t)n\max(\log(m/n),1) + \max(s,t)\log(n+m)$.

\subsection{Complexity model} \label{subsection:complexitymodel}
In computer algebra, it is often implicitly assumed that address and index arithmetics can be performed in constant time. 
This means that reading a byte from any position in the 
memory is done in constant time, and 
reading $k$ bytes is an $O(k)$-operation. 
Such a model has advantages over 
a Turing machine, since it is easier to work with and even 
more realistic in the cases when the memory on a modern computer is enough to 
handle the input data. We will use this model. 

However, when it comes to integer arithmetic (summation and comparisons), the papers
\cite{ Abbottetal, FGLMNC, FGLM, JustStigler07} implicitly assume this to be done in constant time. But with the model used above, the correct
complexities should read $O(\log(a) + \log(b)) = O(\log(\max(a,b)))$ for summation and $O(\log(\min(a,b))$ for comparison for the positive integers $a$ and $b$. 
For simplicity, we will assume that multiplication of the integers is $O(\log(a) \log(b))$, although there are better bounds \cite{furer}.

Inspired by \cite{MMM}, we will split the complexity studies in two parts. We will give arithmetic complexity for the arithmetic operations and time complexity 
for the monomial manipulations. We remark that the time complexity of performing $f$ arithmetic operations always is at least $O(f)$. 
We do not deal with growth of coefficients in the arithmetic operations, 
but refer the reader to \cite{FGLM}. In \cite{Abbottetal}, techniques are discussed when $\Bbbk = \mathbb{Q}$ using the 
Chinese remainder theorem. 

\subsection{Monomial orders} \label{subsection:monomialorders} 
An admissible monomial order $\prec$ on $x_1, \ldots, x_n$ is a total order on the monomials which respects multiplication and 
has the unit $1$ as the minimal element. 
A complete classification of admissible monomial orderings was first given in \cite{Robbiano85}. We will perform complexity analyses for 
a subclass of these orders, namely those defined by $n \times n$-matrices of rational numbers which we define below.
This is a proper restriction, since given an admissible monomial order and a natural number, there is an $n \times n$-matrix of rational numbers
that agrees with the given monomial order on all monomials whose degrees are
bounded by the given number \cite{kuhnlemayr}. In the rest of the paper we will assume that all monomial orders are admissible.  

Let $A = (a_{ij})$ be an element in $GL_n(\mathbb{Q})$ with the property that the first nonzero entry in each column is positive. 
Then we can induce an order $\prec_A$ on the monomials in $S$ by 
$$x^{\alpha} \prec_A x^{\beta} \text{ iff } A\alpha^t < A\beta^t,$$
where  $A\alpha^t < A\beta^t$ is the lexicographic order on $\mathbb{Z}^n$, that is, 
$(v_1, \ldots, v_n) < (w_1, \ldots, w_n)$ if $v_1 = w_1, \ldots, v_{i-1} = w_{i-1}$ and $v_i < w_i$ for some $i$.
Notice that multiplying a row of $A$ by a positive integer does not change the order induced by $A$, hence we may assume that $A$ is an integer matrix. 

Given $A$ and a monomial $x^{\alpha}$, we call the vector $A\alpha^t$ the associated order vector to $x^{\alpha}$ and we denote it by $\ov_A(x^\alpha)$. 
The simple observation
$$\ov_A(x_i x^{\alpha}) = A (\alpha_1, \ldots, \alpha_{i-1}, \alpha_i +1, \alpha_{i+1}, \ldots, \alpha_n)^t= \ov_A(x^\alpha) + (a_{1i}, \ldots, a_{ni}),$$
gives us a handy formula for computing the order vector recursively.

An important subclass of the orders defined above consists of 
(1) the lexicographical order, (2) the degree lexicographical order and (3) the degree reverse lexicographic order. 
These orders, called \emph{standard} in the rest of the paper, are the common most used ones in computer algebra and 
computer algebra systems have them predefined.

To clarify the notion of a matrix representing an order, we show below the matrix representations for the standard orders on 
$x_1, x_2, x_3$, all of which satisfy $x_1 \succ x_2 \succ x_3$. 
$$
A_{lex} = 
\begin{pmatrix}
1 & 0 & 0\\
0 & 1 & 0\\
0 & 0 & 1
\end{pmatrix}, 
A_{deglex} = 
\begin{pmatrix}
1 & 1 & 1 \\
1 & 0 & 0 \\
0 & 1 & 0 
\end{pmatrix},
A_{degrevlex} = 
\begin{pmatrix}
1 & 1 & 1 \\
0 & 0 & -1 \\
0 & -1 & 0 
\end{pmatrix}
$$
If $x_{1} \succ \cdots \succ x_{n}$ for a standard order $\prec$, we
 have $\ov_{lex}(x^\alpha) = \alpha$, $\ov_{deglex}(x^\alpha) = (\sum_i \alpha_i, \alpha_1, \ldots, \alpha_{n-1})$ and 
$\ov_{degrevlex} = (\sum_i \alpha_i, -\alpha_n, \ldots, -\alpha_2)$. 
It is easily seen how to compute $\ov(x_ix^{\alpha})$ given $\ov(x^{\alpha})$ for these orders.

In general, if $x_{i_1} \succ \cdots \succ x_{i_n}$ and $\prec$ is standard, we will assume that the sequence 
$i_1, \ldots, i_n$ is given a priori and that $x^\alpha$ means $x_{i_1}^{\alpha_1} \cdots x_{i_n}^{\alpha_n}$. 
We do not make this assumption when $\prec$ is given by a matrix.
As indicated in the introduction, we will see that 
the complexity analysis of the monomial comparisons is dependent on whether the order is standard or not.

When $\prec$ is a monomial order on $x_1, \ldots, x_n$
it will be useful to restrict $\prec$ to a subset of the variables. If $\Ess$ is such a subset, we write
$\prec_{\Ess}$ to denote the restriction of $\prec$ to $\Ess$.

\subsection{Comparing vectors of integers}
Since we assume that comparing two monomials is the same as lexicographically comparing their associated order vectors, we will
now focus on comparing vectors of integers. 

To be able to prove the next lemma, recall that the number of bits needed to represent an integer $a$ is 

$$
\numbits(a) =  \left\{ \begin{array}{ll}
2 & \text{ if } a = 0\\
\lfloor \log_2(|a|) \rfloor +2   & \text{ otherwise }
\end{array}
\right.,
$$
where one bit is used to represent sign. 
\begin{lemma} \label{lemma:bitsinvector}

Let $\alpha = (\alpha_1, \ldots, \alpha_n)$ be a vector of integers $\alpha_i$.
Let $m = \sum_i |\alpha_i|$. Then 
$$
\sum_i \numbits(\alpha_i) \leq  \left\{ \begin{array}{ll}
 3n  & \text{ if } m \leq 2n\\
n \log_2(m/n)  +2n    & \text{ otherwise }
\end{array}
\right.
$$
\end{lemma}

\begin{proof}
Suppose that $\alpha$ contains $k$ non-zero entries. Without loss of generality, we may assume that
$\alpha_i = 0$ if $i>k$ and $\alpha_i \neq 0$ if $i \leq k$. We get 

$$\sum_j  \numbits(\alpha_{j}) = \sum_{j=1}^k \numbits(\alpha_{j}) + \sum_{j=k+1}^n \numbits(0)$$

$$= \sum_{j=1}^k \lfloor \log_2(\alpha_{j}) \rfloor + 2k + 2(n-k) \leq \sum_{j=1}^k \log_2(\alpha_{j}) +2n.$$
Now $\sum_{j=1}^k \log_2(\alpha_{j}) = \log_2(\alpha_1 \cdots \alpha_k)$ and since
 $$\alpha_1 \cdots \alpha_k \leq \underbrace{m/k \cdots m/k}_{k \text{ times}},$$ 
we conclude that 
$$\sum_j \numbits(\alpha_j) \leq f(k),$$
where $$f(k) = k\log_2(m/k)+2n.$$
We see 
that $f'(k) = 0$ for $k=m/2$ and that $f(m/2)$ is a maximum. If $m \leq 2n$, then $\sum_i(\numbits(\alpha_i)) < f(m/2) = m/2 + 2n \leq 3n$. 
If $2n < m$, then observe that $f(k)$ is a monotone increasing function on the interval $[1,m/2]$ so that 
$\sum_i(\numbits(\alpha_i)) < f(n) = n\log_2(m/n) + 2n$. 
\end{proof}
We can formulate Lemma \ref{lemma:bitsinvector} in a more compact way by saying that the number of bits needed is proportional to $n \max(\log(m/n),1)$.  

\begin{lemma} \label{lemma:cmpvectors}
Let $v = (v_1, \ldots, v_n)$ and $w = (w_1, \ldots, w_n)$ be vectors of integers such that $\sum_i |v_i| =m_1$ and $\sum_i |w_i| =m_2$. Let $m = \max(m_1,m_2)$.
With time complexity  $O(n\max(\log(m/n),1))$ we can determine if $v$ and $w$ differ 
and if they do, we can determine the index $i$ where they differ.
\end{lemma}
\begin{proof}

We proceed as follows. 
Compare $v_1$ and $w_1$. If they differ, we stop and return the index. 
Otherwise, continue until we reach an index $i$ such that $v_i \neq w_i$ or
$v = w$. Let us now analyze the complexity of this procedure.
To compare the vectors $v$ an $w$, we need to compare at most all entries, that is, we get time
complexity proportional to $\sum_i (\numbits(v_i) +\numbits(w_i))$. We rewrite this sum as
$\sum_i \numbits(v_i) + \sum_i \numbits(w_i)$ and  
use Lemma $\ref{lemma:bitsinvector}$ and the remark thereafter to conclude that this sum is dominated by 
$n (\max(\log(m_1/n),1) + \max(\log(m_2/n),1))$.

\end{proof}

\subsection{Comparing monomials with respect to a standard order}

\begin{lemma} \label{lemma:computeovstandard}
Suppose that $\prec$ is a standard order. An upper bound for computing $\ov(x_ix^{\alpha})$ is $O(n\log(m))$, where $m$ is the degree of 
$x^{\alpha}$. 
\end{lemma}
\begin{proof}

To compute $\ov(x^{\alpha})$ is $O(1)$ when $\prec$ is lex, while for a degree order, we need to sum all entries in order to compute $m$. Since all entries are 
bounded by $m$, the lemma follows. 

\end{proof}

\begin{lemma} \label{lemma:standardupdatingop}
Suppose that $\prec$ is a standard order. An upper bound for computing $\ov(x_ix^{\alpha})$ given $\ov(x^{\alpha})$ is $O(\log(m))$, where $m$ is the degree of 
$x^{\alpha}$.
\end{lemma}

\begin{proof}
We need to increment at most two entries (in case of a degree order). The lemma follows since all entries are bounded by $m$. 
\end{proof}

\begin{lemma} \label{lemma:cmpmonomialsstandardgivenop}
An upper bound for the time needed to compare two monomials $x^{\alpha}$ and $x^{\beta}$ given 
$\ov_{\prec}(x^{\alpha})$ and
$\ov_{\prec}(x^{\beta})$
with respect to a standard order is $O(n\max(\log(m/n),1))$, where
$m = \max(\sum_i \alpha_i, \sum_i \beta_i)$.
\end{lemma}

\begin{proof}
Immediately from Lemma \ref{lemma:cmpvectors} when $\prec$ is lex. When $\prec$ is one of the degree orders, the first entry equals the degree and the sum of the 
rest of the entries is bounded by $m$, thus we get the same complexity in this case as well.
\end{proof}

\begin{lemma} \label{lemma:restrictstandard}
Let $\prec$ be a standard monomial order and let $\Ess = \{x_{j_1}, \ldots, x_{j_{\overline{n}}} \}$ 
be a subset of the variables such that $x_{j_1} \succ \cdots \succ x_{j_{\overline{n}}}$. 
To compute $\prec_{\Ess}$ is $O(1)$. 
\end{lemma}
\begin{proof}
Immediate.
\end{proof}

\subsection{Comparing monomials with respect to a matrix order}
\begin{lemma} \label{lemma:computeovmatrix}
An upper bound for the time needed for determining $\ov_A(x^\alpha)$ by computing $A\alpha^t$ 
is proportional to $n^2\max(\log(m/n),1) \log(c)$, where $c = \max(|a_{ij}|)$ and
$m = \sum_i \alpha_i$.
\end{lemma}
\begin{proof}
To determine the $i$'th index of $\ov_A(x^{\alpha})$, one needs to compute
$a_{i1} \alpha_1 + \cdots + a_{in} \alpha_n$. The cost for the multiplication is proportional to
$\sum_i \log_2(c) \log_2(\alpha_i)$ $= \log_2(c)\sum_i \log_2(\alpha_i)$. By an  argument similar 
to the proof of Lemma \ref{lemma:bitsinvector} it follows that this expression is dominated by an expression 
proportional to $n\max(\log(m/n),1) \log(c)$. 
The cost for the addition is negligible and since $i$ runs from $1$ to $n$, the lemma follows.
\end{proof}
\begin{lemma} \label{lemma:computeovgivenovmatrix}
An upper bound for the time needed for computing $\ov_A(x_i x^{\alpha})$ given $\ov_A(x^{\alpha})$ is proportional to
$n\log(cm)$, where $c = \max(|a_{ij}|)$ and $m = \sum_i \alpha_i$.
\end{lemma}
\begin{proof}
From the recursion formulas given in section \ref{subsection:monomialorders}, we  see that
$\ov_A(x^{\alpha} \cdot x_i) =  \ov_A(x^\alpha) + (a_{1i}, \ldots, a_{ni})$, that is, we need to do $n$ summations of integers bounded by $cm$.
\end{proof}

\begin{lemma} \label{lemma:cmpgivenovmatrix}
An upper bound for the time needed to compare two monomials $x^{\alpha}$ and $x^{\beta}$ with respect to a matrix order defined by $A$ 
given $\ov_A(x^{\alpha})$ and $\ov_A(x^{\beta})$ is proportional to $n\log(cm)$, where $c = \max(|a_{ij}|)$ and $m = \sum_i \alpha_i$.
\end{lemma}

\begin{proof}
We have $n$ comparisons of integers bounded by $cm$.
\end{proof}
Let $\prec$ be a monomial order given by a matrix and let $\Varord_{\prec}(n)$ be the cost of determining $i_1, \ldots, i_n$ such that
$x_{i_1} \succ \cdots \succ x_{i_n}$. When 
$\prec$ is a standard order we assume that $i_1, \ldots, i_n$ was given as input, so that $\Varord_{\prec}(n)$ is $O(1)$. However, 
we do not assume this for a general order given by a matrix. Instead we have

\begin{lemma} \label{lemma:varordmatrix}
For a general ordering given by a matrix, we can compute $\Varord_{\prec}(n)$ in time $O(n^2 log(c)\log(n))$.
\end{lemma}
\begin{proof}
To compare $x_i \prec x_j$ is the same as comparing the $i$'th and the $j$'th column of the 
matrix $A$ defining $\prec$. An upper bound for the comparisons is thus 
$O(n \log(nc/n) + n) = O(n\log(c))$. Since sorting is $O(n\log(n))$ comparisons, the upper bound becomes  
$O(n^2 \log(c)\log(n))$.
\end{proof}

\begin{lemma} \label{lemma:restrictmatrix}
Let $\prec$ be a monomial order given by a matrix $A$ and let $\Ess = \{x_{j_1}, \ldots, x_{j_{\overline{n}}} \}$ 
be a subset of the variables such that $x_{j_1} \succ \cdots \succ x_{j_{\overline{n}}}$.  To determine an $\overline{n} \times \overline{n}$-matrix 
$A_{\Ess}$ such that $\prec_{\Ess}$ is given by $A_{\Ess}$ can be done 
using $O(n\overline{n}^2)$ arithmetic operations over $\mathbb{Q}$.
\end{lemma}

\begin{proof}
Clearly the $n \times \overline{n}$ matrix $\overline{A}$ obtained by keeping the columns 
$j_1, \ldots, j_{\overline{n}}$ defines $\prec_{\Ess}$ and it has rank $\overline{n}$. 
Suppose that the $i$'th row of $\overline{A}$ can be written as a linear combination of the rows whose indices are less then $i$. 
Let $\alpha$ and $\beta$ be two order vectors with respect to $\Ess$. Then, if
$\overline{A}\alpha$ and $\overline{A}\beta$ agree on the first $i-1$ rows, then they also agree on the $i$'th row. Hence the $i'$th row is superfluous. Thus, to determine
$A_{\Ess}$ is the same as rowreducing $\overline{A}$, which has arithmetic complexity $O(n\overline{n}^2)$. 
\end{proof}

\subsection{Merging sorted lists of monomials}

Recall that the merging of two sorted lists $a$ and $b$ is done using $O(\max(t,s))$ comparisons, where $t$ is the number of elements in $a$ and $s$ 
is the number of elements in $b$.
When $a$ and $b$ are lists of $n$-tuples of elements in some set and $a$ and $b$ are sorted lexicographically in increasing order, then we can improve the classical 
merge algorithm. We will use this result to analyze the cost of the monomial manipulations in the BM-algorithm.

Let $v = (v_1, \ldots, v_n)$ and $w = (w_1, \ldots, w_n)$ be two $n$-vectors of non-negative integers. If $v\neq w$, 
let $\Delta(v,w)$ be the first index where $v$ and $w$ differ. If $v = w$ let $\Delta(v,w) = n+1$.

\begin{lemma} \label{lemma:lexusingdiff}
If $u<v$, $u<w$ and $\Delta(u,w) < \Delta(u,v)$, then $u< v < w$ and 
$\Delta(v,w) = \Delta(u,w)$.
\end{lemma}

\begin{proof}
Let $k = \Delta(u,w)$. Then $u_1 = w_1, \ldots, u_{k-1} = w_{k-1}$ and 
$u_k < w_k$. Since $k < \Delta(u,v)$, we have $u_1 = v_1, \ldots, u_k = v_k$. Hence
$v_1 = w_1, \ldots, v_{k-1} = w_{k-1}$ and $v_k < w_k$ and thus $v < w$ and 
$\Delta(v,w) = k = \Delta(u,w)$.
\end{proof}

\begin{lemma} \label{lemma:minusingdiff}
If $u<v$, $u<w$ then $\Delta(v,w) \geq \min(\Delta(u,v),\Delta(u,w))$.
\end{lemma}
\begin{proof}
If $\Delta(u,w) < \Delta(u,v)$ or $\Delta(u,v) < \Delta(u,w)$, then the lemma follows by 
Lemma \ref{lemma:lexusingdiff}. Otherwise, $\Delta(u,v) = \Delta(u,w)$ implies that
$v$ and $w$ agree on the first $\Delta(u,v)$ positions.
\end{proof}

\begin{lemma} \label{lemma:merge}
Let $a = (a_1, \ldots, a_t)$ be a list of $n$-tuples of elements in an ordered set $\Sigma$. Suppose that $a$ is sorted lexicographically in increasing order and 
that we are given $\Delta(a_i,a_{i+1})$ for $i = 1, \ldots, t-1$. Let $b$ be any element  in $\Sigma^n$. 
Using $O(t+n)$ comparisons of elements in $\Sigma$ plus time proportional to $t\log(n)$, we may find an index $i$, $0 \leq i \leq t$, such that
$a_1 \leq \cdots \leq a_{i} < b \leq a_{i+1} \leq \cdots \leq a_t$ and $\Delta(a_i,b) $when $i \geq 1$ and $\Delta(b,a_{i+1})$ when $i < t$. 

When $\Sigma$ is the set of non-negative integers and $\sum_i v_i\leq m$ for all $v \in a$ and $\sum_i b_i \leq m$, an upper bound for the time complexity 
is $O(n \max(\log(m/n),1) + t \log(\max(n,m)))$.  
\end{lemma}

\begin{proof}
The proof contains three parts. We (a) give an algorithm, (b) prove its correctness and (c) show that the complexity of the algorithm agrees with what was stated in the
lemma.

\vspace{0.2cm}

\noindent{\textbf{The algorithm}} 

\noindent
At stage $0$, compute $k = \Delta(a_1,b)$. If $a_1 < b$, continue with stage $1$. If $b \leq a_1$ we stop and return $0$ and $\Delta(b,a_1) =k$.

\vspace{0.1cm}
\noindent
At stage $i$, $1\leq i <t$, we suppose that $a_i < b$ and that $\Delta(a_i,b)$ is computed. 

\noindent
If $a_i = a_{i+1}$ (which is equivalent to $\Delta(a_i,a_{i+1})= s+1$), we have $a_i \leq a_{i+1} < b$. Thus, we set $\Delta(a_{i+1},b) = \Delta(a_i,b)$ 
and continue with stage $i+1$.  

\noindent
Else, if $\Delta(a_i,b) > \Delta(a_i,a_{i+1})$ then
$a_i < b < a_{i+1}$ and $\Delta(b,a_{i+1}) = \Delta(a_i,a_{i+1})$ by Lemma \ref{lemma:lexusingdiff}. Thus, we stop and return $i$, 
$\Delta(a_i,b)$ and $\Delta(b, a_{i+1})$ (which equals $\Delta(a_i,a_{i+1})$). 

\noindent
Else, if $\Delta(a_i, b) < \Delta(a_i,a_{i+1})$ then 
$a_i < a_{i+1} < b$ and $\Delta(a_{i+1},b) = \Delta(a_i,b)$, again by Lemma \ref{lemma:lexusingdiff}. We set $\Delta(a_{i+1}, b) = \Delta(a_i,b)$ and continue with stage $i+1$. 

\noindent
Else, if $\Delta(a_i,b) = \Delta(a_i,a_{i+1})=k$, compare $b_k$ and $a_{i+1,k}$, $b_{k+1}$ and $a_{i+1,k+1}$ and so on 
	until we either conclude (1)  $b_{k'} < a_{i+1,k'}$ (2) $b_{k'} > a_{i+1,k'}$ or (3) $b = a_{i+1}$, where in cases (1) and  (2), $k' = \Delta(b, a_{i+1})$.
	In case (1), we have $b < a_{i+1}$. We stop and return $i$, $\Delta(a_i,b)$ and $\Delta(a_{i+1}, b) = k'$.
	In case (2), we have $b > a_{i+1}$. We set $\Delta(a_{i+1}, b) = k'$ and continue with stage $i+1$. 
	In case (3), we have $b = a_{i+1}$. We stop and return $i$, $\Delta(a_i,b)$ and $\Delta(a_{i+1}, b) = t+1$.

\vspace{0.1cm}
\noindent
At stage $t$ we have $a_t<b$. Thus, we stop and return $t$ and $\Delta(a_t,b)$.

\vspace{0.2cm}

\noindent{\textbf{The correctness of the algorithm}} 

\noindent
By construction.

\vspace{0.2cm}

\noindent{\textbf{Complexity of the algorithm}} 

\noindent
There are two key indices that we update during the algorithm. The first ($i$) refers to a position in the list $a$, 
the second ($k$) refers to a position in the vector $b$. Notice that after each comparison, either $i$ or $k$ is increasing. Since both $i$ and $k$ are non-decreasing
it follows that the number of 
comparisons of elements in $\Sigma$ is at most $n+t$. 

The number of $\Delta$-comparisons during the algorithm is one per stage, that is, at most $t$. Every such comparison consists of comparing  
integers bounded by $n$. We conclude that the time needed for the integer comparisons is proportional to $t\log(n)$. 

Let now $\Sigma$ be the set of non-negative integers. Everytime we increase $i$, we make 
a comparison of an integer bounded by $m$, this gives time proportional to $t\log(m)$. However, when increasing $k$, we are in a situation where 
$a_{ik} = b_k$. Hence, the total timed used for the increasings of $k$ is proportional to $n\max(\log(m/n),1)$  by
\ref{lemma:cmpvectors}. Only once during the algorithm we will compare $a_{ik}$ and $b_k$ to conclude that they differ, this cost is $O(\log(m))$.
It follows that an upper bound for the algorithm is proportional to 
$$t\log(n) + t\log(m) + n\max(\log(m/n),1) + \log(m)$$
$$ =t\log(\max(m,n)) + n\max(\log(m/n),1).$$ 

\end{proof}

\begin{theorem} \label{thm:merge}
Let $a=(a_1, \ldots, a_t)$ and $b = (b_1, \ldots, b_s)$ be two lists of $n$-tuples of elements in an ordered set $\Sigma$. Suppose that $a$ and $b$ are sorted lexicographically with respect to the 
order $<$ in $\Sigma$. Suppose that we are given $\Delta(a_i,a_{i+1})$ for $i = 1, \ldots, t-1$ and $\Delta(b_i, b_{i+1})$  for $i=1, \ldots, s-1$.
We can merge $a$ and $b$ into a new list $c$ and compute the sequence $\Delta(c_1,c_2), \Delta(c_2,c_3), \ldots$ 
using $O(s n + t)$ comparisons plus time complexity $O(t \log(n))$. When $\Sigma$ is the non-negative integers and $\sum_i v_i\leq m$ for all $v \in a \cup b$, an upper bound for the time complexity of the algorithm is $O(\min(s,t) n \max(\log(m/n),1) + \max(s,t) \log(n+m))$. 
\end{theorem}

\begin{proof}
Suppose, without loss of generality, that $s < t$.
Let $i_1$ be the index returned after calling the algorithm in Lemma \ref{lemma:merge} with $a$ and $ b_1$. 
Without affecting the complexity of the algorithm, it is clear that we can modify it to give a list $c=(a_1, \ldots, a_{i_1},b_1)$. Suppose that $i_2$ is the index returned after calling the algorithm in Lemma \ref{lemma:merge} with $a_{i_1 +1}, \ldots, a_t$ and $b_2$. Again, it is clear that it is possible to update the list $c$ without affecting the complexity 
so that it reads
$(a_1, \ldots, a_{i_1}, b_1, a_{i_1+1}, \ldots, a_{i_2}, b_2)$. If we proceed in this way we obtain the sequence
$$a_1\leq \cdots\leq a_{i_1}< b_1 \leq a_{i_1 + 1} \leq \cdots \leq a_{i_2}< b_2 \leq \cdots < b_s \leq a_{i_s+1} \leq \cdots \leq a_t.$$
Since the algorithm in Lemma \ref{lemma:merge} returns $\Delta(b_{j}, a_{i_j+1})$ we only need to check the case when $a_{i_j} = a_{i_{j+1}}$ in order to conclude
that $\Delta(c_1,c_{2})$, $\Delta(c_2,c_3), \ldots$ is computed as a side affect of the calls to the modified algorithm. But when 
$a_{i_j} = a_{i_{j+1}}$ we have that $b_{j}$ and $b_{j+1}$ are consecutive and hence $\Delta(b_j,b_{j+1})$ is already computed by assumption. 

Although it does not affect the complexity, stage $0$ of the algorithm given in Lemma \ref{lemma:merge} can be modified. 
When calling with $a_{i_j+1}, \ldots, a_{t}$ and $b_{j+1}$ we can use that $\Delta(b_j,a_{i_j+1})$ already is computed. Indeed,
$\Delta(b_{j+1}, a_{i_j+1}) \geq \min(\Delta(b_j,a_{i_j+1}),\Delta(b_{j}, b_{j+1}))$ by Lemma \ref{lemma:minusingdiff}, so we could call 
the algorithm in Lemma \ref{lemma:merge} with the extra parameter 
$\min(\Delta(b_j,a_{i_j+1}),\Delta(b_{j}, b_{j+1}))$ to speed up the computation of $\Delta(b_{j+1}, a_{i_j+1})$ in stage $0$.

Since we make $s$ calls to the algorithm in Lemma \ref{lemma:merge} (or to be more precise, to the modified algorithm as defined above), 
we make $((i_1+1) + n) + ((i_2+1-(i_1+1) + n) + \cdots + ((i_s+1 - (i_{s-1}+1)) + n) = (i_s+1)+s n < t + sn$ comparisons of elements in $\Sigma$.
The time complexity for the integer comparison part is by Lemma \ref{lemma:merge} proportional to
$(i_1+1) \log(n) + (i_2+1-(i_1+1)) \log(n) + \cdots + (i_s+1 - (i_{s-1}+1)) \log(n) = (i_s+1) \log(n)<t \log(n)$.
If $\Sigma$ is the non-negative integers, then the time complexity of the algorithm becomes
$O(s n \max(\log(m/n),1) + t \log(n+m))$ by Lemma \ref{lemma:cmpvectors}.
\end{proof}

We now give two applications of the new merge algorithm. The idea of the first example is to make the algorithm clear to the reader, while the second example
shows the strength of the algorithm.

\begin{example} \label{example:mergesort}

Suppose that we want to merge the lists 

$$
a=(
x_1 x_3^2 x_4^2, 
x_1x_3^3,
x_1^2x_4, 
x_1^2x_2x_5,
x_1^2x_2x_4^2x_5, 
x_1^3)$$
and
$$ b = (
x_1, 
x_1x_3^2, 
x_1^2x_2x_4x_5, 
x_1^2x_2x_4^2x_5)$$
of monomials, sorted in increasing order with respect to lex and $x_1>x_2>x_3>x_4>x_5$. Since $\ov_{lex}(x^\alpha) = \alpha$, 
we will use the exponent vectors. Thus, in accordance with the notation above, 
$$a_1 = (1,0,2,2,0), a_2 = (1,0,3,0,0), a_3 = (2,0,0,1,0), a_4 = (2,1,0,0,1),$$ 
$$a_5 = (2,1,0,2,1), a_6 = (3,0,0,0,0)$$ and
$$b_1 = (1,0,0,0,0), b_2 = (1,0,2,0,0), b_3 = (2,1,0,1,1), b_4 = (2,1,0,2,1).$$ 

We begin by comparing $b_1$ and $a_1$. We see that $\Delta(b_1,a_1) = 3$ and 
$b_1 < a_1$. Since $\Delta(b_1,b_2) = 3$, we compare $b_{23}$ and $a_{13}$. They are equal, so 
we check the fourth index and conclude that $b_2 < a_1$.  
Now $\Delta(b_2,b_3) = 1$, so we conclude that $a_1<b_3$ only by checking the first index.
Since $\Delta(a_1,a_2) = 3> 1$, we also have $a_2<b_3$. But $\Delta(a_2,a_3) = 1$, so we compare
$a_3$ and $b_3$ from index $1$ and conclude that $a_3 < b_3$  and $\Delta(a_3,b_3)  = 2$. So far we have
$$b_1<b_2<a_1<a_2<a_3.$$
We see that $\Delta(a_3,a_4) = 2$, thus we compare $b_{32}$ and $a_{42}$. They are equal, so are 
$b_{33}$ and $b_{43}$, but $b_{34}>a_{44}$, hence $a_4 < b_3$ and $\Delta(b_3,a_4) = 4$. 
Since also $\Delta(a_4,a_5)=4$, we compare the fourth index of $b_3$ and $a_5$ to conclude that $b_3 < a_5$. 
Finally, since $\Delta(b_3,b_4) = 4$, we check the fourth and fifth indices of $b_4$ and $a_5$ to conclude that 
$b_4 = a_5$. We have 
$$b_1<b_2<a_1<a_2<a_3<a_4<b_3<b_4\leq a_5<a_6.$$
and the sequence of differences is
$$3,4,2,1,2,4,4,6,1.$$ 
\end{example}
%
%
\begin{example} Let $n = 2s$. Let 
$f = x_1x_3+x_2x_s$ and let $g = x_1 x_2 + x_2 x_3 + \cdots + x_{n-1} x_n$. Let $\prec$ be degrevlex with respect to
$x_1 \succ \cdots \succ x_n$. We see that the terms of $f$ and $g$ are written with respect to this order. 
Suppose that, during a Gr\"{o}bner basis computation, we want to compute the S-polynomial of 
$f$ and $g$, that is, we want a sorted expression of 
$S(f,g) = x_2 f - x_3 g$. 
This is the same as merging $x_2^2x_s$ and $(-x_2x_3^2, -x_3^2x_4, -x_3x_4x_5, -x_3x_5x_6, -x_3 x_6x_7, \ldots$, $-x_3 x_{n-2} x_n, -x_3 x_{n-1}x_n)$  
(together with an arithmetic operation in the case of equality).
For simplicity, we write these expressions as lists of order vectors and omit the coefficients. We get 
$$a = (3,\underbrace{0, \ldots, 0}_{ s \text{ times}},-1,0,\ldots, 0, -2)$$ and 
$$b = ((3,0,\ldots,0,-2,-1), (3,0,\ldots,0,-1,-2,0), (3,0,\ldots,0,-1,-1,-1,0),$$ 
$$(3,0,\ldots,0,-1,-1,0,-1,0), \ldots, (3,-1,-1,0,\ldots,0,-1,0)).$$ 
We assume that we are given the sequence of differences for $g$ a priori. Since the sequence of differences is closed under multiplication with a monomial, 
we obtain the sequence of differences for $b$. It reads $(n-2,n-3, \ldots, 2)$. 
Using the algorithm in Lemma \ref{lemma:merge}, we first compare $a$ and $b_1$. After
$s+2$ comparisons, we see that $a \prec b_1$. Since $\Delta(b_1,b_2) = n-2$ and $\Delta(a,b_1) = s+2$, we get (after comparing $n-2$ and $s+2$) 
that $a \succ b_2$ and that $\Delta(a,b_2) = s+2$. Continuing this way we see that 
$a \prec b_1, a \prec b_2$, $a \prec b_3, \ldots, a \prec b_{s-3},$ and $\Delta(a,b_i) = s+2$ for $i = 1, \ldots, s-3$, 
using 
$$s+2 + \underbrace{1 + \cdots + 1}_{ s-3 \text{ times}} = n-1$$ 
comparisons. We have 
$\Delta(b_{s-3}, b_{s-2}) = s+2$ and since also $\Delta(a,b_{s-3}) = s+2$, we compare $b_{s-2,s+2}$ and $a_{s+2}$ to conclude that $a \succ b_{2-2}$. 
In total, we have used $n-1 +2 = 2s+1$ comparisons. 

Proceeding in a naive way we would use $(s+2) \cdot (s-2) = s^2-4$ comparisons.
It should be remarked that one needs extra cost for the bookkeeping of $\Delta(a,b_i)$, for $i = 1, \ldots, s-2$, which is of the same magnitude as the comparisons, 
that is, to be fair, we should compare $2s+1 + s-2 = 3s - 1$ with $s^2-4$. 

\end{example}

\section{The BM-algorithm revised}
We first fix some notation.
If $x^\alpha$ is a monomial, we define its support, which we denote by $\supp(x^\alpha)$, to be the set of all 
$x_i$ such that $\alpha_i > 0$. If $M$ is a set of monomials, then we define the support of $M$ to be the union of the supports of the elements in $M$.
Let $I$ be an ideal in $S$ and let $B$ be any subset of $S$ such that $[B] = \{[b]: b \in B\}$ is a vector space basis for $S/I$. 
Here $[b]$ denotes the equivalence class in $S/I$ containing $b$. 
If $s$ is an element in $S$, its residue can be uniquely expressed as a linear combination of the elements in $[B]$, say $[s] = \sum c_i [b_i]$. 
The $S$-element $\sum c_i b_i$ is then called the normal form of $s$ with respect to $B$ and we write $\nf(s,B) = \sum c_i b_i$. 
We abuse notation and say that $B$ (instead of $[B]$) is a basis for $S/I$.

Let $\prec$ be an (admissible) monomial order. 
The initial ideal of $I$, denoted by $\ini(I)$, is the monomial ideal consisting of
all leading monomials of $I$ with respect to $\prec$. One of the characterizations of a set $G$ being a Gr\"{o}bner basis of an ideal $I$ with respect to a monomial order
$\prec$ is that $G \subseteq I$ and that the leading terms  of $G$ generate $\ini(I)$. An old theorem by Macaulay states that the residues of the 
monomials outside $\ini(I)$ form a $\Bbbk$-basis for the quotient $S/I$. The set of monomials outside $\ini(I)$ is closed under taking submonomials. 
A consequence of this is that we always have the unit $1$ in such a basis (given $\dim_{\Bbbk}(S/I) > 0$). 
Sets of monomials
which are closed under taking submonomials is called
order ideal of monomials, abbreviated by OIM.  

If $p$ is a point in $\Bbbk^n$ and $f$ is an element of $S$, we denote by $f(p)$ the evaluation of $f$ at $p$. 
When $P= \{p_1, \ldots, p_m\}$ is a set of points, $f(P) = (f(p_1), \ldots, f(p_m))$. If $F = \{f_1, \ldots, f_s\}$ is a set
of elements in $S$, then $F(P)$ is defined to be the $s \times m$ matrix whose $i$'th row is $f_i(P)$. 

The vanishing ideal $I(P)$ is the ideal consisting of all elements in $S$ which vanishes on all the points in $P$. 
If $f_1$ and $f_2$ are two elements in $S$ and $[f_1] = [f_2]$ in $S/I(P)$, then $f_1(p) = f_2(p)$ for $p \in P$. 
That a set $[B]$ of $m$ elements is a $\Bbbk$-basis for $S/I$ is equivalent to $\dim_{\Bbbk}(B(P)) = m$.

The BM-algorithm takes as input a set of points in $\Bbbk^n$ and a monomial order. It returns a Gr\"{o}bner basis $G$ of $I$ and the set $B$ of monomials outside $\ini(I)$. 
The BM-algorithm was first given in \cite{BM}. During the years it has been reformulated and modified. In the paper \cite{FGLM}, the ideas of 
the BM-algorithm was used to switch between different Gr\"{o}bner bases of a zero-dimensional ideal. In the unifying paper \cite{MMM}  it was shown that
both the BM- and the FGLM-algorithm can be seen as an algorithm that computes a Gr\"{o}bner basis from an ideal defined by functionals. The complexity studies given 
in \cite{MMM} apply to the BM-algorithm and in fact, \cite{MMM} is by tradition the paper which one refers to when complexity issues of the BM-algorithm
are discussed. However, most of the complexity studies in \cite{MMM} are done by referring to the paper \cite{FGLM}. 

We will first discuss the complexity studies of the BM-algorithm and postpone the connection with ideals defined by functionals
to section \ref{subsec:FGLM}. 

\subsection{Two formulations of the BM-algorithm}

We first try to describe the ideas of the BM-algorithm without going into details. 
The algorithm uses a list $L$ of possible basis-candidates, a list $G$ of the partial Gr\"{o}bner basis and a list
$B$ of the partial basis-elements. In the main loop, one checks if 
$l(P) \in \spann_{\Bbbk}(B(P))$, where $l$ is the least element in $L$ with respect to $\prec$. If it was, write $l$ as a linear combination
$\sum_i c_ib_i$ of the elements in $B$ and insert $l -  \sum_i c_i b_i$ into $G$. Otherwise, insert $l$ into $B$ and 
update $L$ with new basis candidates.

We give below the formulation of the BM-algorithm as given in \cite{Abbottetal}.
When the algorithm terminates, $G$ is the Gr\"{o}bner basis with respect to $\prec$ and $B$ is the
complement of the initial ideal with respect to $\prec$.

\begin{itemize}
\item[\textbf{C1}] Start with empty lists $G = B = R = [~]$ a list $L = [1]$, and a matrix 
$C = (c_{ij})$ over $\Bbbk$ with  $m$ columns and initially zero number of rows. 

\item[\textbf{C2}] If $L = [~]$, return the pair $[G,B]$ and stop. Otherwise, choose the monomial 
$t = \min_{\prec}(L)$, the smallest according to the ordering $\prec$. Delete $t$ from $L$.

\item[\textbf{C3}] Compute the evaluation vector $(t(p_1), \ldots, t(p_m)) \in \Bbbk^m$, and reduce it against
the rows of $C$ to obtain 
$$(v_1, \ldots, v_m) = (t(p_1), \ldots, t(p_m)) - \sum_i a_i (c_{i1}, \ldots, c_{im}) \hspace{1.5cm} a_i \in \Bbbk.$$

\item[\textbf{C4}] If $(v_1, \ldots, v_m) = (0, \ldots, 0)$, then append the polynomial $t - \sum_i a_i r_i$ to the list $G$,
where $r_i$ is the $i$'th element of $R$. Continue with step C2.
l
\item[\textbf{C5}] If $(v_1, \ldots, v_m) \neq (0, \ldots, 0)$, then add $(v_1, \ldots, v_m)$ as a new row to $C$, and 
$t - \sum_i a_i r_i$ as a new element to $R$. Append the power product $t$ to $B$, and add to $L$ those elements of 
$\{x_1t, \ldots, x_nt\}$ which are neither multiples of an element of $L$ nor of $\ini(G)$. Continue with step C2.
\end{itemize}
The authors in \cite{Abbottetal} claims that this is the same as the algorithm restricted to the BM-situation which appeared in \cite{MMM}, 
but this is not exactly the case. 
To get the algorithm given in \cite{MMM} restricted to the BM-situation using the five-step description, we need to reformulate steps 2 and 5.

\begin{itemize}

\item[\textbf{C2'}] If $L = [~]$, return the pair $[G,B]$ and stop. Otherwise, choose the power product 
$t = \min_{\prec}(L)$, the smallest according to the ordering $\prec$. Delete $t$ from $L$.
If $t$ is a multiple of an element in $\ini(G)$, then repeat this step. Else, continue with step C3.

\item[\textbf{C5'}] If $(v_1, \ldots, v_m) \neq (0, \ldots, 0)$, then add $(v_1, \ldots, v_m)$ as a new row to $C$, and 
$t - \sum_i a_i r_i$ as a new element to $R$. Append the power product $t$ to $B$, and 
merge $\{x_1t, \ldots, x_nt\}$ and $L$. Continue with step C2'.

\end{itemize}

\begin{lemma}
The output from the two algorithms given above agree and $G$ will be a reduced Gr\"{o}bner basis for I(P) and $[B]$ will be a basis for $S/I(P)$. 
\end{lemma}
\begin{proof}
It was proved \cite{MMM} that using steps C2' and C5', $G$ will be a reduced Gr\"{o}bner basis and $[B]$ will be a basis for $S/I$. To prove that the two algorithms
agree, one needs to show that the elements which are multiples of an element in $L$ and not inserted into the list $L$ during stage C5 are not needed in the computations.
We now sketch this argument. Suppose that 
$x_it$ is an element which is a multiple of an element in $L$ during the algorithm using steps C2 and C5 and suppose also that 
$x_it$ is in the complement of $\ini(I)$. Then all submonomials of 
$x_it$ is in $B$. Suppose that $u$ is the largest. Clearly $|u| = |x_it|-1$. Thus, $x_it = x_ju$ for some $x_j$. At some stage during the algorithm using 
steps C2 and C5, $u$ will be the least element in $L$ and accordingly, no submonomials of $x_it$ will be left in $L$. Thus, $x_ju = x_it$ will be added to $L$ in step
C5'.
\end{proof}

These two variants give the same arithmetic complexity, which is reported \cite{MMM,FGLM,Abbottetal} to be $O(nm^3)$ arithmetic operations. 
We now give a better bound.

\begin{proposition} \label{prop:complexityarithmeticBM}
The arithmetic complexity of the algorithms given above agree and an upper bound is $O(nm^2+\min(m,n)m^3)$.
\end{proposition}

\begin{proof}
The arithmetic operations are performed in steps C3 and C4. Step C3 involves an evaluation and a row reduction. To compute 
$(t(p_1), \ldots, t(p_m))$ requires $m$ multiplications, since $t = x_it'$ for some $x_i$ and some $t'$ and 
$t'(p_1), \ldots, t'(p_m)$ already has been evaluated. The row reduction requires 
$O(m^2)$ arithmetic operations. Step C4 consist of expressing $\sum_i a_i r_i$ in the basis and is also an $O(m^2)$ operation. 
Notice that the number of calls to C4 is exactly $|B| = m$, the number of calls to C5/C5' is exactly $|G|$ and thus, the number of calls to C3 is exactly 
$m + |G|$. Thus, the number of arithmetic operations is proportional to $(m + |G|) \cdot m^2)$, so it is enough to prove that $|G|$ is proportional to $n+\min(m,n)m$.

The number of variables in $B$ is at most $\min(n,m-1)$ since the basis consist of $m$ elements and $1 \in B$. 
Accordingly, if we let $s$ denote the number of variables in $\ini(G)$, then $n-\min(n,m-1) \leq s \leq n$ ($s = n$ only if $m =1$).
Now notice that if $x_it$ is an element inserted into $L$, then, since $x_i<x_it$ and $x_i$ 
is inserted into $L$ during the first step of the algorithm, $x_i$ will be treated before $x_it$. This shows that  
when we add a new element $t$ to $B$ and $\{x_1t, \ldots, x_nt\}$ to $L$, we know that at most $n-s$ of these elements 
would be added to $\ini(G)$, because if $x_i \in \ini(G)$, then $x_it \in \ini(G)$.

Since we add $m$ elements to $B$, we see that the elements of degree more than one in $\ini(G)$ is at most $m(n-s)$.
We conclude that the number of elements in $\ini(G)$ is at most $m(n-s) + s< m\min(n,m-1) + n$.

\end{proof}
 
The parts of the algorithms that concern monomial manipulations 
are harder to analyze. As stated in the introduction, the number of integer comparisons in the monomial manipulation-part 
is reported to be proportional to $n^2m^2$, assuming a standard order. We agree that this is an upper bound for the algorithm using steps
C2' and C5'. Since it is not explained in \cite{Abbottetal} how to check if $t$ is a multiple of an element in $L$ or of $\ini(G)$, 
the actual behavior of this algorithm might be worse than the algorithm using step C2' and C5', although the list $L$ during the former algorithm 
contains less elements than the list $L$ during the latter. 

To check if $t$ is a multiple of an element in $\ini(G)$ using step C2' and C5' is simple and is due to the following nice observation given in \cite{FGLM}.

\begin{lemma}
Using steps C2' and C5', to check if $t$ is a multiple of $\ini(G)$ can be replaced by 
checking if $|\supp(t)|> \Occ(t)$, where $\Occ(t)$ denote the number of copies of $t$ in $L$.
\end{lemma}
\begin{proof}
Let $t$ be the first element in $L$. There exists exactly $\supp(t)$ submonomials of $t$ of degree $|t|-1$. All copies of $t$ comes from a 
submonomial of of $t$ of degree $|t|-1$ in $B$, and each submonomial gives rise to exactly one copy. 
It follows that if there exists a submonomial of $t$ not in $B$ then $|\supp(t)|> \Occ(t)$ and vice versa. But the first statement is equivalent with $t$ being a multiple of
an element in $\ini(G)$.
\end{proof}
We remark that this simplification of the check does not apply when using steps C2 and C5.


\subsection{Optimizing the BM-algorithm} \label{sec:optbm}

In order to improve the time complexity during the  monomial manipulations part, we will use a projection technique together with Theorem \ref{thm:merge}. 
The idea is to identify a set $\Ess$ of the variables with the property that $\supp(B) \subseteq \Ess$ and $|\Ess|\leq \min(m-1,n)$. Once the set $\Ess$ is identified, we only
consider monomials in the monoid generated by $\Ess$, hence the associated order vectors for the monomials is of length bounded by
$\min(m-1,n)$. 

The projection technique is covered in the following lemma and is in some sense inspired by \cite{JustStigler06,JustStigler07}.

\begin{lemma} \label{lemma:suppB}
Let $\prec$ be a monomial order and let 
$p_1, \ldots, p_m$ be distinct points in $\Bbbk^n$. 
Then we can determine a subset $\Ess$ of $\{x_1, \ldots, x_n\}$ with the following properties
\begin{itemize}
\item  $\supp(B) \subseteq \Ess$
\item  $|\Ess| \leq \min(m-1,n)$
\item $\dim_{\Bbbk}(\Ess(P)) = |\Ess|$ ( $\Ess(P)$ has full rang.)
\item $x_k - \sum_j c_{kj} x_{i_j} \in I(P)$ where $x_k \succ x_{i_j}$ if $c_{kj} \neq 0$.
\end{itemize}
 in $O(nm^2)$ arithmetic operations plus time complexity $\Varord_{\prec}(n)$. 
\end{lemma}

\begin{proof}
We first determine $i_1, \ldots, i_n$ such that $x_{i_1} \succ \cdots \succ x_{i_n}$ which has time complexity $\Varord_{\prec}(n)$.
Let
$$
E_j =  \left\{ \begin{array}{ll}
\{\} & \text{ if } j = 0\\
E_{j-1} \cup \{x_{i_{n-j+1}}\}  & \text{ if } x_{i_{n-j}}(P) \notin \spann_{\Bbbk}\{1(P),E_{j-1}(P) \}\\
E_{j-1} & \text{ otherwise }
\end{array}
\right.
$$
and let $\Ess = E_n.$ 
It is clear that if $x_k \notin \Ess$, then $x_k \in \ini(I)$, since $x_k$ can be written as a linear combination of smaller elements, hence
$\supp(B) \subseteq \Ess$. By a dimension argument, we have $|\Ess| \leq m-1$ and since $\Ess$ is a subset of the variables, clearly $|\Ess| \leq n$.
By construction $\dim_{\Bbbk}(\Ess(P)) = |\Ess|$.
To determine if $x_{i_{n-j+1}}(P) \notin \spann_{\Bbbk}\{1(P),E_{j-1}(P)$ 
is $O(m^2)$ arithmetic operations by using a matrix representation. This is repeated $n$ times, which gives the arithmetic complexity $O(nm^2)$.
Finally when $x_k \notin \Ess$ we obtain a the expression $x_k = \sum_j c_{kj} x_{i_j}$ mod $I(P)$ as a side effect of the matrix representation.
\end{proof}

Let $\pi$ be the natural projection from 
$\Bbbk^n$ to $\Bbbk^{\overline{n}}$ with respect to $\Ess = \{x_{i_1}, \ldots, x_{i_{\overline{n}}}\}$, that is,  
$\pi((a_1, \ldots, a_n)) = (a_{i_1}, \ldots, a_{i_{\overline{n}}})$. Let $\pi^*$ be the corresponding monomorphism
from $T = \Bbbk[y_{i_{1}}, \ldots, y_{i_{\overline{n}}}]$ to $S$ given by $y_{i_j} \mapsto x_{i_j}$. If $f$ is any element in 
$T$, then by construction $\pi^*(f)(p) = f(\pi(p))$. This shows that if 
$\Ess(P)$ has full rang, so has $(\pi^*)^{-1}(\Ess)(\pi(P))$. It follows that the points in $\pi(P)$ are distinct. 
This leads us to use the following isomorphism result.

\begin{lemma} \label{lemma:iso}
Let $p_1, \ldots, p_m$ be distinct points i $\Bbbk^n$. Let $I$ be the vanishing ideal with respect to these points. 
Let $\pi$ be a projection from $\Bbbk^n$ to $\Bbbk^{\overline{n}}$ such that $\pi(p_1), \ldots, \pi(p_m)$ are distinct. 
Let $T = \Bbbk[y_{i_1}, \ldots y_{i_{\overline{n}}}]$ and let $J$ be the vanishing ideal with respect to $\pi(p_1), \ldots, \pi(p_m).$
Then $S/I$ and $T/J$ are isomorphic as algebras.
\end{lemma}

\begin{proof}
Let $\pi^*$ be the map defined above. For $f \in T$ we have $\pi^*(f)(p) = f(\pi(p))$. Notice that $f \in J$ is equivalent to 
$f(\pi(q_i)) = 0, \forall i$, which is equivalent
to $\pi^*(f)(p_i) = 0, \forall i$, which is equivalent to $\pi^*(f) \in I$. This allows us to extend  $\pi^*$ to a monomorphism from $T/J$ to $S/I$. Since 
$\pi(p_1), \ldots, \pi(p_m)$ are distinct, we have $\dim_{\Bbbk} (T/J) = \dim_{\Bbbk} (S/I)$ and thus, the extension of $\pi^*$ is an isomorphism of algebras.
\end{proof}

\begin{lemma} \label{lemma:isoB}
Let $\prec$ be a monomial order on $S$. Let $\Ess$ be 
a subset of the variables such that $\supp(B) \subseteq \Ess$, $|\Ess| \leq \min(m-1,n)$ and $\dim_{\Bbbk}(\Ess(P)) = |\Ess|$, 
where $B$ is the monomials outside $\ini(I(P))$ 
with respect to $\prec$. Let $\pi$ be the projection defined by $\pi((a_1, \ldots, a_n)) = (a_{i_1}, \ldots, a_{i_{\overline{n}}})$ and let 
$\pi^*$ be the corresponding monomorphism from 
$T = \Bbbk[y_{i_{1}}, \ldots, y_{i_{\overline{n}}}]$ to $S$. Let $\prec'$ be the monomial order defined by 
$y^{\alpha} \prec' y^{\beta}$ if $\pi^*(y^{\alpha}) \prec \pi^*(y^{\beta})$. Let $B' = \{b_1', \ldots, b_m'\}$ be the set of monomials outside
$\ini(I(\pi(P))$ with respect to $\prec'$. Then $\pi^*(B') = B$. 
\end{lemma}
\begin{proof}
Suppose that $\pi^*(b_i')$ can be written as a linear combination of elements in $B$; $\pi^*(b_i') = \sum_j c_j b_j$ with 
$\pi^*(b_i') \succ b_j$ if $c_j \neq 0$. Since $\supp(B) \subseteq \Ess$, we get $b_i' = \sum_j c_j (\pi^*)^{-1}(b_j)$ with 
$b_i' \succ' (\pi^*)^{-1}(b_j) $, which is a contradiction. Hence $\pi^*(B') \subseteq B$ from which it follows 
that $\pi^*(B') = B$ by Lemma \ref{lemma:iso} and a dimension argument.
\end{proof}

\begin{lemma} \label{lemma:isoGB}
In the context of Lemma \ref{lemma:suppB} and Lemma \ref{lemma:isoB}, suppose that $x_k - \sum_j c_{kj} x_{i_j} \in I(P)$ for all $x_k$ outside $\Ess$. Let $G'$ be a 
reduced Gr\"{o}bner basis for $I(\pi(P))$. Then  $$G = \pi^*(G') \sqcup \{x_k - \sum_{j} c_{kj} \pi^* (\nf((\pi^*)^{-1}(x_{i_j}),G'))\}$$ is a reduced Gr\"{o}bner basis for $I(P)$. 
\end{lemma} 

\begin{proof}
Since $\pi^*(B')$ is the complement of $\ini(I(P))$ by Lemma \ref{lemma:isoB}, it follows that 
$\ini(I(P))$ is minimally generated by $\Ess^c \sqcup \pi^{*}(\ini(I(\pi(P))))$. Clearly 
$\pi^*(G')$ is contained in $G$. Thus, it is enough to prove 
that 
$$x_k - \sum_{j} c_{kj} \pi^* (\nf((\pi^*)^{-1}(x_{i_j}),G')) \in I(P)$$ 
and that 
$x_k$ is larger than any monomial occurring in the right hand sum. Since 
$x_k - \sum_{j} c_{kj} x_{i_j} \in I(P)$, we have that 
$$x_k - \sum_{j} c_{kj} (\pi^*) (\nf((\pi^*)^{-1}(x_{i_j}),G')) \in I(P)$$ is equivalent to 
$$\sum_{j} c_{kj} x_{i_j}  - \sum_{j} c_{kj} (\pi^*) (\nf((\pi^*)^{-1}(x_{i_j}),G')) \in I(P).$$ Using the monomorphism
$(\pi^*)^{-1}$ we see that this is equivalent to 
$$\sum_{j} c_{kj} ((\pi^*)^{-1} (x_{i_j}) - \nf((\pi^*)^{-1}(x_{i_j}),G') \in I(\pi(P)).$$
Since each term $(\pi^*)^{-1} (x_{i_j}) - \nf((\pi^*)^{-1}(x_{i_j}),G')$ is in $G'$, we are done with the first part. 
The second part follows since $x_k \succ x_{i_j}$ if $c_{kj} \neq 0$ and each $\pi^*(x_{i_j})$ is written as a linear combination of elements less than 
$\pi^*(x_{i_j})$ with respect to $\prec'$. 
\end{proof}

We now give an example of our method.

\begin{example}
Consider the points 
$$p_1 = (1,1,0,1,0), p_2 = (2,2,1,1,1), p_3 = (2,0,1,1,-1), p_4 = (5,3,4,1,2)$$ 
in $\mathbb{Q}^5$, with respect to pure lex and
$x_1 \succ x_2 \succ x_3 \succ x_4 \succ x_5$. Using Lemma \ref{lemma:suppB}, we get $\Ess = \{x_3,x_5\}$ and 
$$x_4 = 1, x_2 = x_5+1, x_1 = x_3+1,$$ 
everything mod $I(P)$. Thus, let $\pi(a_1, \ldots, a_5) = (a_3,a_5)$ and 
let $T = \Bbbk[y_1,y_2]$ and let $\pi^*$ be defined by $y_1 \mapsto x_3$ and $y_2 \mapsto x_5$. We have that $\prec'$ is lex with
$y_1 \succ' y_2$ and 
$$\pi(P)=\{(0,0), (1,1), (1,-1), (4,2)\}.$$
A call to the BM-algorithm with $\pi(P)$ and $\prec'$ yields $B' = \{1,y_2,y_2^2,y_2^3\}$ as the set of monomials outside $\ini(I(\pi(P)))$ and
$$y_2^4 +2y_2 - y_2^2-2y_2^3,  y_1 - y_2^2$$ 
as a Gr\"{o}bner basis $G'$ for $I(\pi(P))$. Thus, a Gr\"{o}bner basis $G$ for 
$I(P)$ is 
$$\{x_5^4  + 2x_5 - x_5^2-2x_5^3, x_3 - x_5^2,$$ 
$$x_4-\pi^*(\nf((\pi^*)^{-1}(1),G')),$$
$$x_2-\pi^*(\nf((\pi^*)^{-1}(x_5),G'))-\pi^*(\nf((\pi^*)^{-1}(1),G')),$$
$$x_1-\pi^*(\nf((\pi^*)^{-1}(x_3),G'))-\pi^*(\nf((\pi^*)^{-1}(1),G'))\}$$ 
and the complement of $\ini(I(P))$ is 
$$B = \{(\pi^*)^{-1}(1), (\pi^*)^{-1}(y_2), (\pi^*)^{-1}(y_2^2), (\pi^*)^{-1}(y_2^3)\} = \{1, x_5, x_5^2, x_5^3\}.$$ 
We have $$\pi^*(\nf((\pi^*)^{-1}(1),G')) = \pi^*(\nf(1,G')) = \pi^*(1) = 1,$$
$$\pi^*(\nf((\pi^*)^{-1}(x_5),G') = \pi^*(\nf(y_2,G')) = \pi^*(y_2)=x_5$$ and
$$\pi^*\nf((\pi^*)^{-1}(x_3),G')) = \pi^*(\nf(y_1,G')) = \pi^*(y_2^2) = x_5^2.$$ Thus 
$$G = \{x_5^4 +2x_5 - x_5^2-2x_5^3, x_3 - x_5^2, x_4-1,x_2-x_5-1,x_1-x_5^2-1\}.$$
\end{example}
Notice that although $\Ess = \{x_3,x_5\}$ were linearly independent with respect to $P$, it did not follow that $\Ess \subset B$.

To analyze the complexity of the method above, we first determine the cost of the monomial manipulations. 
\begin{proposition} \label{prop:timecplmonomialmanstandard}
Let $\prec$ be a standard order. An upper bound for the time complexity of the monomial manipulation part of the BM-algorithm using the projection technique is 
$$O(\min(m,n)m^2 \log(m))$$ 
\end{proposition}
\begin{proof}
Suppose that $m<n$. We can now use the projection technique described above so that 
$\overline{n}\leq m-1$. Since the projection technique only affects the arithmetic complexity, we do not need to consider the cost for it in this analysis. 
By Lemma \ref{lemma:restrictstandard}, to determine $\prec_{\Ess}$ is $O(1)$. 
Everytime we insert an element into $B$, we insert at most $\overline{n}$ elements into $L$. Thus, the number of elements in $L$ is bounded by
$\overline{n}m$. Thus, the complexity of the monomial manipulation part is dominated by 
merging a sorted list of $\overline{n}$ monomials with a sorted list of at most $\overline{n}m$ monomials, repeated $m$ times.
To compute $\ov(x_i x^{\alpha})$ given $\ov(x^{\alpha})$ is $O(\log(m))$ 
by Lemma \ref{lemma:standardupdatingop}.  Thus, each time an element is inserted into $B$, it is an
$O(\overline{n}\log(m))$-operation to create the list of monomials which we will merge with $L$. 
Since we create at most $m$ such lists, the total time needed for creation is $O(\overline{n}m\log(m))$. Using 
Theorem \ref{thm:merge} we see that each merge has time complexity
$O(\overline{n}^2 \max(\log(m/\overline{n}),1) + \overline{n}m \log(\overline{n} + m)).$ 

If $m \geq n$, then all arguments hold if we replace $\overline{n}$ by $n$. Thus, in general, each merge has time complexity
$$O(\min(m,n)^2 \max(\log(m/\min(m,n)),1) + \min(m,n)m \log(\min(m,n) + m))$$
which equals $$O(\min(m,n) m \log(m))$$ by a straightforward calculation.
Since there are exactly $m-1$ merges, we get the complexity
$$O(\min(m,n)m^2\log(m) + \min(m,n)m\log(m)),$$
where the last term comes from the creation process and is negligible.
\end{proof}

\begin{proposition} \label{prop:timecplmonmialmanmatrix}
Let $\prec$ be an order defined by an integer matrix $A$. Let $c = \max(|a_{ij}|)$. We give two upper bounds for the time 
 complexity of the monomial manipulation part of the BM-algorithm using the projection technique, based on two different methods. 
When $m \geq n$, the methods agree and an upper bound is 
$$O(n^2 m \log(cm) + n m^2 \log(n+m) + n^2m\log(cm) +  n^2\log(c)\log(n)).$$

When $m<n$, the first method has the bound
$$O(m^3 \log(cm) + n^2\log(c)\log(n))$$ to which one needs to add the cost for 
$O(nm^2)$ arithmetic operations over $\mathbb{Q}$. 

When $m<n$, the second method has the bound
$$O(nm^2 \log(cm) + n^2\log(c)\log(n)).$$

\end{proposition}
\begin{proof}
First of all we need to determine 
$\Varord_{\prec}$, an $n^2\log(c)\log(n)$-operation by Lemma \ref{lemma:varordmatrix}.
Suppose that $m<n$. We can use the projection technique described above and we now have two choices. 
Either we use Lemma \ref{lemma:restrictmatrix} to construct an $\overline{n} \times \overline{n}$-matrix $A_{\Ess}$ using $O(nm^2)$ 
arithmetic operations over $\mathbb{Q}$, or we can use the $n \times \overline{n}$-submatrix of $A$, where we keep the columns that refers to the variables
in $\Ess$. 

In the first case, the cost for computing $\ov(x_im)$ given $\ov(m)$ is $O(\overline{n}\log(cm))$ by Lemma \ref{lemma:computeovgivenovmatrix}, 
so the total time needed for the construction of the associated order vectors is $O(\overline{n}^2m\log(cm))$.
By Lemma \ref{lemma:cmpgivenovmatrix} and Theorem \ref{thm:merge} we see that each merge has time complexity
$$O(\overline{n}^2 \log(cm) + \overline{n} m\log(\overline{n}+m)) = O(m^2 \log(cm))$$ as we merge a list of $\overline{n}$ elements with at most 
$\overline{n}m$ elements. Since we make $m-1$ merges, we deduce that the overall time complexity of the first method is
$$O(m^3 \log(cm) + \overline{n}^2m\log(cm) + n^2\log(c)\log(n))$$
$$  = O(m^3 \log(cm) + n^2\log(c)\log(n))$$ to which we need to add 
$O(nm^2)$ arithmetic operations over $\mathbb{Q}$. 

The second method differs from the first in that the vectors are $n$-tuples rather than $\overline{n}$-tuples. Thus, 
computing $\ov(x_im)$ given $\ov(m)$ is an $n \log(cm)$ operation, so the total time needed for the construction process is $O(n\overline{n}m\log(cm))$. Each merge 
requires time $O(n\overline{n} \log(cm) + \overline{n} m\log(\overline{n}+m)) = O(nm \log(cm))$, so the overall time complexity of the second method becomes
$$O(nm^2 \log(cm) + n\overline{n}m\log(cm) + n^2\log(c)\log(n))$$
$$  = O(nm^2 \log(cm) + n^2\log(c)\log(n)).$$

When $m \geq n$, we do not need to project and the two methods agree. The cost for the construction of the associated order vectors becomes 
$O(n^2m\log(cm))$, each merge is $n^2 \log(cm) + n m \log(n+m)$ and thus an upper bound is
$$O(n^2 m \log(cm) + n m^2 \log(n+m) + n^2m\log(cm) +  n^2\log(c)\log(n)).$$ 
\end{proof}
We are ready to state the main theorem.

\begin{theorem} \label{thm:main}
An upper bound for the arithmetic complexity of the BM-algorithm using steps C1, C2', C3, C4 and C5' and the projection technique based on Lemma \ref{lemma:suppB} 
is $$O(nm^2 + \min(m,n)m^3).$$ 
To this we need to add the time complexity
$$O(\min(m,n) m^2 \log(m))$$ 
when $\prec$ is standard. 

When $\prec$ is given by a matrix there are two methods to use. When $m \geq n$, the methods agree and an upper bound is 
$$O(n^2 m \log(cm) + n m^2 \log(n+m) + n^2m\log(cm) +  n^2\log(c)\log(n)).$$

When $m<n$, the first method has the bound
$$O(m^3 \log(cm) + n^2\log(c)\log(n))$$ to which one needs to add the cost for 
$O(nm^2)$ arithmetic operations over $\mathbb{Q}$. 

When $m<n$, the second method has the bound
$$O(nm^2 \log(cm) + n^2\log(c)\log(n)).$$
\end{theorem}

\begin{proof}
The complexity result for the monomial manipulation part is given in Proposition \ref{prop:complexityarithmeticBM}. If $m<n$ and we use the projection technique
described in Lemma \ref{lemma:suppB}, we will get the set $B'$ of monomials outside $\ini(I(\pi(P)))$ and a Gr\"{o}bner basis $G'$ for 
$\ini(I(\pi(P)))$. By Lemma \ref{lemma:isoB}, $B = \pi^*(B')$ and by Lemma \ref{lemma:isoGB}, 
$$G = \pi^*(G') \sqcup \{x_k - \sum_{j} c_{kj} \pi^* (\nf((\pi^*)^{-1}(x_{i_j}),G'))\}.$$
The computation of $\pi^*(B)$ is immediate, since the only computation needed is changing indices. 
The arithmetic complexity bound follows if we can show that each $x_k - \sum_{j} c_{kj} \pi^* (\nf((\pi^*)^{-1}(x_{i_j}),G'))$
is computable within $O(m^2)$ arithmetic operations. To get a short proof, we will not use the information
$x_k -  \sum_{j} c_{kj} x_{i_j} \in I$. Instead we compute the evaluation vector 
$(x_k(p_1), \ldots, x_k(p_m))$ $= (p_{1k}, \ldots, p_{mk})$ and write it as a linear combination of the elements in $B$, an operation which requires 
$O(m^2)$ arithmetic operations. Since $B$ is a basis, 
the linear combination will then equal $\sum_{j} c_{kj} \pi^* (\nf((\pi^*)^{-1}(x_{i_j}),G'))$.

The complexity for the monomial manipulation part follows from Proposition \ref{prop:timecplmonomialmanstandard} and Proposition
\ref{prop:timecplmonmialmanmatrix}.

\end{proof}

The following corollary states that our version of the BM-algorithm is prefarable to the EssGB-algorithm \cite{JustStigler07}.

\begin{corollary}
When $m<n$ and the order is standard, the BM-algorithm using steps C1, C2', C3, C4 and C5' and the projection technique based on Lemma \ref{lemma:suppB} 
has arithmetic complexity 
$$O(nm^2 + m^4).$$ 
To this we need to add the time complexity
$$O(m^3\log(m)).$$
\end{corollary}

%

\subsection{Applications to the FGLM-algorithm and for ideals defined by functionals} \label{subsec:FGLM}

As was noticed in \cite{MMM}, both the FGLM- and the BM-algorithm are instances of definitions of an ideal defined by means of a finite set of functionals 
$L_i : \Bbbk[x_1, \ldots, x_n] \to \Bbbk$, such that $I$ is in the kernel of $\Psi: \Bbbk[x_1, \ldots, x_n] \to \Bbbk^m, \Psi(f) = L_1(f), \ldots, L_m(f)$. 
For the BM-setting, the functionals are defined by $L_i(f) = f(p_i)$ and in the FGLM-setting, the functionals are defined by 
$\nf(f,G_1) = \sum L_i(f)e_i$. In \cite{MMM}, a list of different problems that can be seen as instances of an ideal defined by functionals is given.

If we use the steps C1,C2',F3,C4,C5' of the BM-algorithm where F3 is defined below, we obtain 
Algorithm 1 in \cite{MMM}.

\hspace{1cm}

\textbf{F3} Compute $\Psi(t)=(b_1, \ldots, b_m)$  and reduce it against the rows of $C$ to obtain
$$(v_1, \ldots, v_m) = (b_1, \ldots, b_m) - \sum_i a_i (c_{i1}, \ldots, c_{im}) \hspace{1.5cm} a_i \in \Bbbk.$$

\hspace{1cm}

It is reported in \cite{MMM} Theorem 5.1, that Algorithm 1 in \cite{MMM} needs $O(nm^3+fnm^2)$ arithmetic operations, where $f$ denotes the cost of evaluating a 
functional. However, as for the BM-algorithm, one can replace the term $nm^3$ by $m^4 + nm^2$. Since we have shown (Proposition \ref{prop:complexityarithmeticBM}) 
that the number of calls to $C3$ equals $|G| + m = n + \min(m,n)m + m$, one can replace the term $fnm^2$ by $fnm + f\min(m,n)m^2 + fm$. Thus, Algorithm 1 in \cite{MMM} uses
$O(\min(m,n)m^3 + nm^2 + fnm+f\min(m,n)m^2)$ arithmetic operations. 
In the BM- or the FGLM-situation, it is shown in \cite{MMM} that $f=1$, by a recursive argument. 

To the arithmetic complexity one needs to add the cost for the monomial manipulations, which is the same as for the BM-algorithm, since it is clear
that we can use the projection technique described in Lemma \ref{lemma:suppB} if we replace 
$$x_{i_{n-j}}(P) \notin \spann_{\Bbbk}\{1(P),E_{j-1}(P) \}$$ by
$$\Psi(x_{i_{n-j}}) \notin \spann_{\Bbbk} \{\Psi(1), \Psi(x_{i_{n-j+1}}), \ldots, \Psi(x_{i_{n}})\}.$$

It follows that Theorem \ref{thm:main} can be lifted to the general setting of ideals defined by functionals.
We state

\begin{theorem}
An upper bound for the arithmetic complexity of Algorithm 1 in \cite{MMM} using the projection technique is 
$$O(\min(m,n)m^3 + nm^2 + fnm+f\min(m,n)m^2).$$ 

To this we need to add the time complexity
$$O(\min(m,n) m^2 \log(m))$$
when $\prec$ is standard.

When $\prec$ is given by a matrix there are two methods to use. When $m \geq n$, the methods agree and an upper bound is 
$$O(n^2 m \log(cm) + n m^2 \log(n+m) + n^2m\log(cm) +  n^2\log(c)\log(n)).$$

When $m<n$, the first method has the bound
$$O(m^3 \log(cm) + n^2\log(c)\log(n))$$ to which one needs to add the cost for 
$O(nm^2)$ arithmetic operations over $\mathbb{Q}$. 

When $m<n$, the second method has the bound
$$O(nm^2 \log(cm) + n^2\log(c)\log(n)).$$

\end{theorem}

\section{Discussion and future work}
The projection idea can be used for ideals defined by projective points and also 
in the non-commutative versions of the FGLM-algorithm, given in \cite{FGLMNC}. It should be remarked that the
monomial manipulations with respect to non-commuting variables is computationally harder than those for commuting variables.

Future work involves the question: Why do we need a Gr\"{o}bner basis for ideals defined by vanishing points? 
In the biological applications, where one is primary interested in normal form computations, 
it seems enough to compute a set $B$  such that $[B]$ is a $\Bbbk$-basis for $S/I$. 

\end{document}